\documentclass[12pt,a4paper]{article}
\usepackage[T1]{fontenc}
\usepackage[utf8]{inputenc}
\usepackage{lmodern}
\usepackage{amsmath,amssymb,amsthm,mathtools}
\usepackage{geometry}
\usepackage{enumitem}
\usepackage{xcolor}
\usepackage{hyperref}
\usepackage{microtype}
\usepackage{booktabs}
\hypersetup{colorlinks=true,linkcolor=blue!45!black,citecolor=blue!45!black,urlcolor=blue!45!black}
\allowdisplaybreaks[2]
\theoremstyle{plain}
\newtheorem{theorem}{Theorem}[section]
\newtheorem{lemma}{Lemma}
\theoremstyle{remark}
\newtheorem{remark}{Remark}

\newcommand{\Rd}{\mathbb R^d}
\newcommand{\Rn}{\mathbb R^n}
\newcommand{\Sn}{S^{n-1}}
\newcommand{\calM}{\mathcal M}
\newcommand{\calG}{\mathcal G}
\newcommand{\norm}[2][]{\left\|#2\right\|_{#1}}

\newcommand{\eps}{\varepsilon}

\title{\bfseries A Density-Distance Version of the Carlen--Frank--Lieb Stability Theorem}
\author{Gangsong Leng\\[-0.2em]
\small East China Normal University; Shanghai University\\[-0.2em]
\small \texttt{lenggangsong@163.com}}
\date{}

\begin{document}
\maketitle

\begin{abstract}
Carlen, Frank and Lieb studied stability estimates for the lowest eigenvalue of a Schr\"odinger operator by decomposing the problem into a stability estimate for H\"older's inequality and a stability estimate for a Gagliardo--Nirenberg--Sobolev inequality. In this note we point out that, if the H\"older step is replaced by the optimal $L^1$-stability theorem of Leng and Lu in probabilistic form, then one obtains a density-distance version of the Carlen--Frank--Lieb stability theorem. The new formulation measures the $L^1$ distance between the normalized density $V_-^s/\int V_-^s$ induced by the negative part of the potential and the corresponding density induced by an optimal potential, where $s=\gamma+d/2$. As a geometric application of the same idea, we also derive a density-stability version of the $L_p$ mixed volume inequality. In the case where one of the two convex bodies is centrally symmetric and both bodies are trapped between two concentric Euclidean balls, this gives an averaged stability estimate for the non-evenness of the support function.
\end{abstract}



\section{Introduction}

Stability problems ask whether, when an inequality is almost sharp, the objects involved must be quantitatively close to the equality cases. In \cite{CFL}, Carlen, Frank and Lieb studied stability estimates for the lowest eigenvalue of the Schr\"odinger operator
\[
        -\Delta+V
\]
in $L^2(\Rd)$. Let
\[
\lambda(V)=\inf\left\{\int_{\Rd}\bigl(|\nabla \psi|^2+V|\psi|^2\bigr)\,dx:
        \psi\in H^1(\Rd),\ \norm[2]{\psi}=1\right\}.
\]
Put
\[
        s=\gamma+\frac d2,
        \qquad
        \frac1s+\frac2Q=1,
        \quad \text{that is,}\quad Q=\frac{2s}{s-1}.
\]
Under the assumptions of Carlen--Frank--Lieb, namely $\gamma>1/2$ for $d=1$ and $\gamma>0$ for $d\ge2$, define the Keller-type sharp constant
\[
C_{\gamma,d}:=\sup_V
\frac{|\lambda(V)|}{\left(\int_{\Rd}V_-^s\,dx\right)^{1/\gamma}}.
\]
Carlen, Frank and Lieb first proved that the family of optimal potentials exists and is unique up to translations and dilations. We denote this family by
\[
\calM=
\left\{W\in L^s(\Rd):
 |\lambda(W)|=C_{\gamma,d}\left(\int_{\Rd}W_-^s\,dx\right)^{1/\gamma}
\right\}.
\]
Their main stability theorem may be summarized as follows.

\begin{theorem}[Carlen--Frank--Lieb stability theorem]\label{thm:CFL}
Let $\gamma>1/2$ if $d=1$, and let $\gamma>0$ if $d\ge2$. Put $s=\gamma+d/2$. Then there is a constant $c_{\gamma,d}>0$ such that, for every $V\in L^s(\Rd)$ with $V_-\not\equiv0$, the following stability estimates hold.

If $s\le2$, then
\[
\frac{|\lambda(V)|}{\left(\int_{\Rd}V_-^s\,dx\right)^{1/\gamma}}
\le
C_{\gamma,d}
\left(
1-c_{\gamma,d}\inf_{W\in\calM}
\frac{\norm[s]{V_- -W_-}^2}{\norm[s]{V_-}^2}
\right).
\tag{1.1}
\]
If $s\ge2$ and $s'=s/(s-1)$, then
\[
\frac{|\lambda(V)|}{\left(\int_{\Rd}V_-^s\,dx\right)^{1/\gamma}}
\le
C_{\gamma,d}
\left(
1-c_{\gamma,d}\inf_{W\in\calM}
\frac{\norm[s']{V_-^{s-1}-W_-^{s-1}}^2}{\norm[s']{V_-^{s-1}}^2}
\right).
\tag{1.2}
\]
Here one may restrict the infimum to optimal potentials $W$ with the same normalization scale as $V$.
\end{theorem}

The proof in \cite{CFL} has a transparent structure. The lowest eigenvalue problem is first reduced to a two-variable variational problem, and the deficit is then decomposed into a deficit coming from H\"older's inequality and a deficit coming from a Gagliardo--Nirenberg--Sobolev inequality. There are several natural stability forms for H\"older's inequality. The version of Aldaz is related to a Hellinger-type distance, while Carlen--Frank--Lieb use a stability estimate expressed through the duality map in $L^p$. In the present note we use the optimal $L^1$-stability theorem of Leng and Lu in probabilistic form. This leads to the following density-distance version of the Carlen--Frank--Lieb theorem.

For $V_-\not\equiv0$, define the probability density induced by the negative part of $V$ by
\[
        P_V(x)=\frac{V_-(x)^s}{\int_{\Rd}V_-^s\,dx}.
\]
Our first main result is the following.

\begin{theorem}[Density-distance version of the CFL stability theorem]\label{thm:main}
Let $\gamma>1/2$ if $d=1$, and let $\gamma>0$ if $d\ge2$. Put $s=\gamma+d/2$. Then there is a constant $c_{\gamma,d}>0$ such that, for every $V\in L^s(\Rd)$ with $V_-\not\equiv0$,
\[
\frac{|\lambda(V)|}{\left(\int_{\Rd}V_-^s\,dx\right)^{1/\gamma}}
\le
C_{\gamma,d}
\left[
1-c_{\gamma,d}\inf_{W\in\calM}
\left(\int_{\Rd}|P_V-P_W|\,dx\right)^2
\right].        \tag{1.3}
\]
\end{theorem}

The main difference between Theorem \ref{thm:main} and Theorem \ref{thm:CFL} is the choice of distance. The original CFL theorem measures the distance from $V_-$, or a power of $V_-$, to the optimal family in suitable $L^p$-type norms. Theorem \ref{thm:main} measures instead the $L^1$ distance between the probability densities induced by the negative parts of the potentials. Thus the two theorems have parallel roles in spectral theory: the original result gives stability in a function-norm sense, while Theorem \ref{thm:main} gives stability in a mass-distribution sense. The latter formulation is especially well suited to links with surface area measures, cone-volume measures and densities arising from support functions in geometric analysis.

As an application of this density point of view, we next give a density-stability form of the $L_p$ mixed volume inequality. In the centrally symmetric case, under a uniform concentric ball condition, it further yields an averaged estimate for the non-evenness of the support function. We use standard notation from convex geometry; see Schneider's monograph \cite{Schneider}. The basic theory of $L_p$ mixed volumes goes back to Lutwak's $L_p$ Brunn--Minkowski--Firey theory \cite{Lutwak}. Let $K,L$ be convex bodies in $\Rn$ containing the origin in their interiors. Let $h_K,h_L$ be their support functions and let $S_K$ be the surface area measure of $K$. For $p>1$, define the $L_p$ mixed volume by
\[
        V_p(K,L)=\frac1n\int_{\Sn} h_L(u)^p h_K(u)^{1-p}\,dS_K(u).
\]
We also define the normalized cone-volume measure of $K$ by
\[
        d\mu_K(u)=\frac{h_K(u)}{n|K|}\,dS_K(u).
\]
If $|K|=|L|$, put
\[
        \delta_p(K,L)=\frac{V_p(K,L)}{|K|}-1.
\]
By the $L_p$ Minkowski inequality, $\delta_p(K,L)\ge0$.

\begin{theorem}[Density stability for $L_p$ mixed volumes and a centrally symmetric consequence]\label{thm:geom}
Let $K,L\subset\Rn$ be convex bodies containing the origin in their interiors. Let $p>1$, let $q=p/(p-1)$, and assume that $|K|=|L|$. Put
\[
        r(u)=\frac{h_L(u)}{h_K(u)}.
\]
Then
\[
\int_{\Sn}
\left|
\frac{r(u)^p}{1+\delta_p(K,L)}-1
\right|\,d\mu_K(u)
\le
\sqrt{2q\,\delta_p(K,L)}.               \tag{1.4}
\]
If, in addition, $K=-K$, then
\[
\int_{\Sn}
\left|
\left(\frac{h_L(u)}{h_K(u)}\right)^p
-
\left(\frac{h_L(-u)}{h_K(u)}\right)^p
\right|\,d\mu_K(u)
\le
2(1+\delta_p(K,L))\sqrt{2q\,\delta_p(K,L)}.    \tag{1.5}
\]
If, moreover, there are constants $0<m\le M<\infty$ such that
\[
        mB_2^n\subset K,L\subset MB_2^n,
\]
then
\[
\int_{\Sn}|h_L(u)-h_L(-u)|\,d\mu_K(u)
\le
\frac{2M(1+\delta_p(K,L))}{p}
\left(\frac{M}{m}\right)^{p-1}
\sqrt{2q\,\delta_p(K,L)}.              \tag{1.6}
\]
In particular, if
\[
        L_0=\frac12(L+(-L)),
\]
then $L_0$ is centrally symmetric and
\[
\int_{\Sn}|h_L(u)-h_{L_0}(u)|\,d\mu_K(u)
\le
\frac{M(1+\delta_p(K,L))}{p}
\left(\frac{M}{m}\right)^{p-1}
\sqrt{2q\,\delta_p(K,L)}.              \tag{1.7}
\]
\end{theorem}

Theorem \ref{thm:geom} says that, when the $L_p$ mixed volume inequality is close to equality, the probability measure induced by $h_L^p h_K^{1-p}\,dS_K$ is close to the cone-volume measure of $K$. If $K$ is centrally symmetric, this further implies that the support function of $L$ is close to an even function in the average sense with respect to $\mu_K$.

\section{Proofs of the Theorems}

This section proves Theorems \ref{thm:main} and \ref{thm:geom}. All independent ingredients needed in the proofs are stated as lemmas. We first give the H\"older stability tool used in this note.

\begin{lemma}[Leng--Lu inequality]\label{lem:lenglu}
Let $(X,\mathcal A,\mu)$ be a measure space. Let $p>1$, $q>1$, and
\[
        \frac1p+\frac1q=1.
\]
If $f,g\ge0$ and
\[
        \int_X f\,d\mu=\int_X g\,d\mu=1,
\]
then
\[
1-\int_X f^{1/p}g^{1/q}\,d\mu
\ge
\frac1{2pq}
\left(\int_X |f-g|\,d\mu\right)^2.       \tag{2.1}
\]
The constant $1/(2pq)$ is best possible.
\end{lemma}

\begin{proof}
We give the proof for completeness. First, we establish a strengthened form of Young's inequality. Set
\[
        A=\frac{2-1/p}{3},\qquad B=\frac{1+1/p}{3}.
\]
Then $A+B=1$, and for all $a,b>0$,
\[
\frac ap+\frac bq-a^{1/p}b^{1/q}
\ge
\frac1{2pq}\frac{(a-b)^2}{Aa+Bb}.        \tag{2.2}
\]
By homogeneity, it is enough to take $b=1$ and write $a=x$. Thus it remains to prove
\[
F(x):=2pq(Ax+B)\left(\frac xp+\frac1q-x^{1/p}\right)-(x-1)^2\ge0.
\]
A direct computation gives $F(1)=F'(1)=0$ and
\[
F''(x)=\frac{2(p+1)}{3p(p-1)}x^{1/p-2}
\left(px^{2-1/p}-(2p-1)x+(p-1)\right).
\]
By the weighted AM--GM inequality,
\[
\frac{p}{2p-1}x^{2-1/p}+\frac{p-1}{2p-1}\ge x,
\]
so the expression in parentheses is nonnegative. Hence $F$ is convex, and since $x=1$ is a stationary point, $F(x)\ge0$. This proves (2.2).

Applying (2.2) pointwise, we have
\[
\frac fp+\frac gq-f^{1/p}g^{1/q}
\ge
\frac1{2pq}\frac{(f-g)^2}{Af+Bg},
\]
where the right side is interpreted as $0$ on the set $\{f=g=0\}$. Integration gives
\[
1-\int f^{1/p}g^{1/q}\,d\mu
\ge
\frac1{2pq}\int\frac{(f-g)^2}{Af+Bg}\,d\mu.
\]
By Cauchy's inequality,
\[
\int\frac{(f-g)^2}{Af+Bg}\,d\mu
\ge
\frac{\left(\int |f-g|\,d\mu\right)^2}{\int(Af+Bg)\,d\mu}
=
\left(\int |f-g|\,d\mu\right)^2.
\]
This proves (2.1). The sharpness of the constant follows from the two-point distributions
\[
 f=\left(\frac12,\frac12\right),\qquad
 g=\left(\frac{1+\eps}{2},\frac{1-\eps}{2}\right)
\]
by letting $\eps\to0$. Indeed, the left side is
\[
1-\frac12(1+\eps)^{1/q}-\frac12(1-\eps)^{1/q}
=\frac1{2pq}\eps^2+O(\eps^3),
\]
whereas $\|f-g\|_1^2=\eps^2$.
\end{proof}

\begin{lemma}[Density stability for the H\"older step]\label{lem:holderpart}
Let $s>1$ and $Q=2s/(s-1)$. If $U\ge0$, $\int_{\Rd}U^s=1$, and $\psi\in L^Q(\Rd)$ with $\psi\not\equiv0$, then
\[
\left(\int_{\Rd}|\psi|^Q\,dx\right)^{2/Q}-\int_{\Rd}U\psi^2\,dx
\ge
\frac{\norm[Q]{\psi}^2}{sQ}
\left(
\int_{\Rd}
\left|U^s-\frac{|\psi|^Q}{\norm[Q]{\psi}^Q}\right|dx
\right)^2.      \tag{2.3}
\]
\end{lemma}

\begin{proof}
Set
\[
        f=U^s,
        \qquad
        g=\frac{|\psi|^Q}{\norm[Q]{\psi}^Q}.
\]
Then $f$ and $g$ are probability densities. Apply Lemma \ref{lem:lenglu} with the conjugate exponents $s$ and $Q/2$, since
\[
        \frac1s+\frac2Q=1.
\]
We obtain
\[
1-\int_{\Rd}U\frac{\psi^2}{\norm[Q]{\psi}^2}\,dx
\ge
\frac1{sQ}
\left(
\int_{\Rd}
\left|U^s-\frac{|\psi|^Q}{\norm[Q]{\psi}^Q}\right|dx
\right)^2.
\]
Multiplying by $\norm[Q]{\psi}^2$ gives the assertion.
\end{proof}

The next lemma rewrites the GNS stability theorem of Carlen--Frank--Lieb in a density-distance form. Let
\[
E[\psi]=\int_{\Rd}|\nabla\psi|^2\,dx-
\left(\int_{\Rd}|\psi|^Q\,dx\right)^{2/Q},
\]
and let $\calG$ be the family of optimizers for the normalized GNS minimization problem. Carlen--Frank--Lieb proved that there exists a constant $a_{Q,d}>0$ such that, for $\norm[2]{\psi}=1$,
\[
E[\psi]+C_{\gamma,d}
\ge
 a_{Q,d}\inf_{\phi\in\calG}\|\psi-\phi\|_{H^1}^2.       \tag{2.4}
\]
This is the core GNS stability input used below from \cite{CFL}.

\begin{lemma}[A density consequence of GNS stability]\label{lem:gns-density}
Under the same assumptions as in \cite{CFL}, there is a constant $b_{Q,d}>0$ such that, for every $\psi\in H^1(\Rd)$ with $\norm[2]{\psi}=1$,
\[
E[\psi]+C_{\gamma,d}
\ge
 b_{Q,d}\inf_{\phi\in\calG}
\left(
\int_{\Rd}
\left|
\frac{|\psi|^Q}{\norm[Q]{\psi}^Q}
-
\frac{|\phi|^Q}{\norm[Q]{\phi}^Q}
\right|dx
\right)^2.      \tag{2.5}
\]
\end{lemma}

\begin{proof}
Write
\[
        D(\psi)=\inf_{\phi\in\calG}\|\psi-\phi\|_{H^1},
        \qquad
        \rho_\psi=\frac{|\psi|^Q}{\norm[Q]{\psi}^Q}.
\]
The CFL GNS stability estimate gives
\[
        E[\psi]+C_{\gamma,d}\ge a_{Q,d}D(\psi)^2.       \tag{2.6}
\]
All functions in $\calG$ have the same $L^Q$ norm, say $m_Q>0$. By Sobolev embedding there is a constant $S_{Q,d}$ such that
\[
        \|u\|_Q\le S_{Q,d}\|u\|_{H^1}.
\]
Choose $\varepsilon>0$ so small that $S_{Q,d}\varepsilon\le m_Q/2$. If $D(\psi)<\varepsilon$ and $\phi\in\calG$ satisfies $\|\psi-\phi\|_{H^1}<\varepsilon$, then
\[
        \norm[Q]{\psi}\ge m_Q/2,
        \qquad
        \norm[Q]{\psi}\le m_Q+S_{Q,d}\varepsilon.
\]
Thus, in this neighborhood, $\norm[Q]{\psi}$ and $\norm[Q]{\phi}$ have uniform positive lower and upper bounds. Using
\[
        \bigl||a|^Q-|b|^Q\bigr|
        \le Q(|a|^{Q-1}+|b|^{Q-1})|a-b|
\]
and H\"older's inequality, we get
\[
        \left\||\psi|^Q-|\phi|^Q\right\|_1
        \le C_{Q,d}\|\psi-\phi\|_Q
        \le C_{Q,d}\|\psi-\phi\|_{H^1}.
\]
Since the denominators have a uniform positive lower bound and
\[
        |\norm[Q]{\psi}^Q-\norm[Q]{\phi}^Q|
        \le \left\||\psi|^Q-|\phi|^Q\right\|_1,
\]
it follows that
\[
        \|\rho_\psi-\rho_\phi\|_1
        \le C'_{Q,d}\|\psi-\phi\|_{H^1}.       \tag{2.7}
\]
Taking the infimum over $\phi\in\calG$, we obtain
\[
        \inf_{\phi\in\calG}\|\rho_\psi-\rho_\phi\|_1^2
        \le C''_{Q,d}D(\psi)^2 .              \tag{2.8}
\]
Combining this with (2.6) proves the lemma when $D(\psi)<\varepsilon$.

If $D(\psi)\ge\varepsilon$, then the $L^1$ distance between two probability densities is at most $2$, and hence
\[
        \inf_{\phi\in\calG}\|\rho_\psi-\rho_\phi\|_1^2\le 4
        \le \frac4{\varepsilon^2}D(\psi)^2.
\]
Combining this again with (2.6) and decreasing the constant if necessary, we obtain (2.5) for all $\psi$.
\end{proof}

\begin{lemma}[Density control of the CFL deficit]\label{lem:combine}
With the notation above, there is a constant $c_{Q,d}>0$ such that, for every $U\ge0$ with $\int U^s=1$ and every $\psi\in H^1(\Rd)$ with $\norm[2]{\psi}=1$,
\[
E[\psi]+H[\psi,U]+C_{\gamma,d}
\ge
c_{Q,d}
\inf_{\phi\in\calG}
\left(
\int_{\Rd}
\left|
U^s-\frac{|\phi|^Q}{\norm[Q]{\phi}^Q}
\right|dx
\right)^2,       \tag{2.9}
\]
where
\[
        H[\psi,U]=\norm[Q]{\psi}^2-\int_{\Rd}U\psi^2\,dx.
\]
\end{lemma}

\begin{proof}
Put
\[
        A=U^s,
        \qquad
        B=\frac{|\psi|^Q}{\norm[Q]{\psi}^Q},
        \qquad
        C_\phi=\frac{|\phi|^Q}{\norm[Q]{\phi}^Q}.
\]
Let
\[
        \Delta=E[\psi]+H[\psi,U]+C_{\gamma,d}.
\]
Since $H[\psi,U]\ge0$ and $E[\psi]+C_{\gamma,d}\ge0$, we have $\Delta\ge0$.

Let $m_Q=\norm[Q]{\phi}$ for $\phi\in\calG$; this value is independent of $\phi$ and is positive. Choose $\eta>0$ so small that
\[
        E[\psi]+C_{\gamma,d}<\eta
\]
implies that the $H^1$ distance from $\psi$ to $\calG$ is sufficiently small, and hence
\[
        \norm[Q]{\psi}\ge m_Q/2.               \tag{2.10}
\]
Such an $\eta$ exists by the CFL GNS stability estimate. If $\Delta<\eta$, then $E[\psi]+C_{\gamma,d}\le\Delta<\eta$, so (2.10) holds. In this case Lemma \ref{lem:holderpart} gives
\[
        H[\psi,U]
        \ge \frac{m_Q^2}{4sQ}\|A-B\|_1^2.      \tag{2.11}
\]
At the same time, Lemma \ref{lem:gns-density} gives
\[
        E[\psi]+C_{\gamma,d}
        \ge b_{Q,d}\inf_{\phi\in\calG}\|B-C_\phi\|_1^2.      \tag{2.12}
\]
By the triangle inequality,
\[
        \|A-C_\phi\|_1^2
        \le 2\|A-B\|_1^2+2\|B-C_\phi\|_1^2.
\]
Taking the infimum over $\phi\in\calG$ and using (2.11), (2.12), we find that
\[
        \Delta
        \ge c\inf_{\phi\in\calG}\|A-C_\phi\|_1^2
\]
for some $c>0$ depending only on $Q$ and $d$.

If $\Delta\ge\eta$, then the squared $L^1$ distance on the right side of (2.9) is at most $4$. Thus, after replacing $c$ by $\min\{c,\eta/4\}$, the same estimate remains valid. This proves (2.9).
\end{proof}

\begin{proof}[Proof of Theorem \ref{thm:main}]
We use the variational reduction of \cite{CFL}, combined with Lemma \ref{lem:combine}. For clarity, first assume that $V\le0$ and $\int V_-^s=1$, and write $U=V_-$. The CFL reduction gives
\[
        -C_{\gamma,d}
        =\inf_{\norm[2]{\psi}=1,\,\int U^s=1}
        \bigl(E[\psi]+H[\psi,U]\bigr),
\]
and the optimal potentials are related to GNS optimizers by
\[
        W_-=\frac{|\phi|^{Q-2}}{\norm[Q]{\phi}^{Q-2}}.
\]
Since
\[
        s=\frac{Q}{Q-2},
\]
we have
\[
        W_-^s=\frac{|\phi|^Q}{\norm[Q]{\phi}^Q}.
\]
Thus Lemma \ref{lem:combine} shows that the CFL deficit controls the square of the $L^1$ distance from the normalized density $U^s$ to the density induced by an optimal potential.

For a general $V$, set
\[
        U=\frac{V_-}{\left(\int V_-^s\right)^{1/s}}.
\]
Using the homogeneity and scaling normalization in \cite{CFL}, and then returning to the original variables, we get
\[
C_{\gamma,d}-
\frac{|\lambda(V)|}{\left(\int V_-^s\right)^{1/\gamma}}
\ge
C_{\gamma,d}c_{\gamma,d}
\inf_{W\in\calM}\left(\int |P_V-P_W|\right)^2.
\]
This is equivalent to (1.3).
\end{proof}

We now prove the convex-geometric application.

\begin{lemma}[$L_p$ mixed volume normalization]\label{lem:geom-density}
Under the assumptions of Theorem \ref{thm:geom}, let
\[
        r(u)=\frac{h_L(u)}{h_K(u)},
        \qquad
        A=\int_{\Sn}r(u)^p\,d\mu_K(u).
\]
Then
\[
        A=\frac{V_p(K,L)}{|K|}=1+\delta_p(K,L).
\]
Moreover,
\[
\int_{\Sn}r(u)\,d\mu_K(u)=\frac{V_1(K,L)}{|K|}\ge1.
\]
\end{lemma}

\begin{proof}
The first identity follows directly from the definitions:
\[
\int r^p\,d\mu_K
=
\frac1{n|K|}\int h_L^p h_K^{1-p}\,dS_K
=
\frac{V_p(K,L)}{|K|}.
\]
The second identity is similar. By the classical first Minkowski inequality and the assumption $|K|=|L|$,
\[
        V_1(K,L)\ge |K|^{(n-1)/n}|L|^{1/n}=|K|.
\]
Hence $\int r\,d\mu_K\ge1$.
\end{proof}

\begin{proof}[Proof of Theorem \ref{thm:geom}]
Let
\[
        A=1+\delta_p(K,L)=\int r^p\,d\mu_K.
\]
Define two probability densities with respect to the probability measure $\mu_K$ by
\[
        F=\frac{r^p}{A},\qquad G=1.
\]
Applying Lemma \ref{lem:lenglu} to $F$ and $G$ gives
\[
1-\int F^{1/p}\,d\mu_K
\ge
\frac1{2pq}\left(\int |F-1|\,d\mu_K\right)^2.
\]
Since $F^{1/p}=r/A^{1/p}$, Lemma \ref{lem:geom-density} yields
\[
1-\int F^{1/p}\,d\mu_K
=1-\frac{\int r\,d\mu_K}{A^{1/p}}
\le
1-A^{-1/p}.
\]
Also,
\[
        1-(1+\delta)^{-1/p}\le \frac{\delta}{p}
        \qquad(\delta\ge0).
\]
Therefore
\[
\int_{\Sn}\left|\frac{r^p}{A}-1\right|\,d\mu_K
\le
\sqrt{2q\,\delta_p(K,L)},
\]
which proves (1.4).

If $K=-K$, then $h_K(u)=h_K(-u)$, and both $S_K$ and $\mu_K$ are even measures. Hence, by (1.4) and the change of variables $u\mapsto -u$,
\[
\int\left|\frac{r(-u)^p}{A}-1\right|\,d\mu_K(u)
\le
\sqrt{2q\,\delta_p(K,L)}.
\]
The triangle inequality gives
\[
\int\left|\frac{r(u)^p}{A}-\frac{r(-u)^p}{A}\right|\,d\mu_K(u)
\le
2\sqrt{2q\,\delta_p(K,L)}.
\]
Multiplying by $A=1+\delta_p(K,L)$ and using $h_K(u)=h_K(-u)$, we obtain (1.5).

Assume now that
\[
        mB_2^n\subset K,L\subset MB_2^n.
\]
Then, for every $u\in\Sn$,
\[
        m\le h_K(u),h_L(u)\le M,
\]
and hence
\[
        \frac mM\le r(u),r(-u)\le \frac Mm.
\]
By the mean value theorem, for $a,b\in[m/M,M/m]$,
\[
        |a^p-b^p|
        \ge
        p\left(\frac mM\right)^{p-1}|a-b|.
\]
Consequently,
\[
|h_L(u)-h_L(-u)|
=h_K(u)|r(u)-r(-u)|
\le
\frac{M}{p}\left(\frac Mm\right)^{p-1}|r(u)^p-r(-u)^p|.
\]
Integrating this inequality and using (1.5), we obtain (1.6). Finally,
\[
        h_{L_0}(u)=\frac12\bigl(h_L(u)+h_L(-u)\bigr),
\]
so
\[
        |h_L(u)-h_{L_0}(u)|=\frac12|h_L(u)-h_L(-u)|.
\]
Thus (1.7) follows immediately from (1.6).
\end{proof}

\begin{remark}
The conclusion of Theorem \ref{thm:geom} is stated with respect to the cone-volume measure $\mu_K$. The concentric ball condition alone does not imply that $\mu_K$ is comparable with the ordinary spherical measure. If one assumes, in addition, that $K$ is sufficiently smooth and has curvature bounded above and below, then corresponding averaged estimates with respect to the spherical measure can also be obtained.
\end{remark}

\end{document}